\documentclass[10pt]{article}   
\usepackage{amsfonts,amsmath,latexsym}
\usepackage[T1]{fontenc}
\usepackage{dsfont}
\usepackage{hyperref}
\usepackage[dvips]{graphicx}
\usepackage{epsfig}
\usepackage{enumerate}
\usepackage{epsf}    
\usepackage{amsthm}
\usepackage{color}
\usepackage{amssymb}

\usepackage{caption}
\usepackage{subfigure}

\usepackage{fancyhdr} 

\usepackage{palatino}

\newtheorem{thm}{Theorem}[section]
\newtheorem{prop}[thm]{Proposition}
\newtheorem{conclusion}[thm]{Conclusion}

\newtheorem{defi}[thm]{Definition}
\newtheorem{rem}[thm]{Remark}

\newtheorem{ass}{Assumption}

\def\R{\mathbb R}
\def\N{\mathbb N}

\def\E{\mathbb E}
\def\P{\mathbb P}
\def\Q{\mathbb Q}

\def\shc{{\cal C}}

\def\she{{\cal E}}
\def\shf{{\cal F}}
\def\shg{{\cal G}}
\def\shh{{\cal H}}
\def\shm{{\cal M}}

\def\halb{{\frac{1}{2}}}

\author{
{\sc Lucas IZYDORCZYK}
\thanks{ENSTA-Paris, Institut Polytechnique de Paris.
Unit\'e de Math\'ematiques Appliqu\'ees (UMA).
 E-mail:{ \tt lucas.izydorczyk@ensta-paris.fr}} 
{\sc,}\ {\sc Nadia OUDJANE}
\thanks{EDF R\&D,   and FiME (Laboratoire de Finance des March\'es de l'Energie
(Dauphine, CREST,  EDF R\&D) www.fime-lab.org). 
E-mail:{\tt  
nadia.oudjane@edf.fr}}
\ {\sc and}\ {\sc Francesco RUSSO} 
\thanks{ENSTA Paris, Institut Polytechnique de Paris.
Unit\'e de Math\'ematiques Appliqu\'ees (UMA). 
 E-mail:{\tt  francesco.russo@ensta-paris.fr}.
 }}

\date{December 2019}

\title{McKean Feynman-Kac probabilistic representations of non-linear partial differential equations}

\newcommand{\MBFigure}[6]{
$\left. \right.$ \\
\refstepcounter{figure}
\addcontentsline{lof}{figure}{\numberline{\thefigure}{\ignorespaces #5}}
\begin{center}
\begin{minipage}{#1cm}
\centerline{\includegraphics[width=#2cm,angle=#3]{#4}}
\begin{center}
\upshape{F\textsc{ig} \normal
\end{center}
size{\thefigure}. $-$} #5
\end{center}
\label{#6}
\end{minipage}
\end{center}
$\left. \right.$ \\}

\begin{document}
\maketitle
 \begin{abstract} This paper presents a partial state of the art about the topic
of representation of generalized Fokker-Planck Partial Differential Equations (PDEs) by solutions of McKean Feynman-Kac Equations (MFKEs) that generalize the notion of McKean Stochastic Differential Equations (MSDEs).  
%MSDEs are Stochastic Differential Equations (SDEs) whose coefficients do not only depend on time or on position of the particle solution, but also on its law, 
While MSDEs can be related to non-linear Fokker-Planck PDEs, MFKEs can be related to non-conservative non-linear PDEs. Motivations come
from modeling issues but also from numerical approximation issues
in computing the solution of a PDE, arising for instance in the context of stochastic control.
MFKEs also appear naturally in representing final value problems related to backward Fokker-Planck equations. 
%in formulating the time reversal of a diffusion (considered as an independentproblem)  and in solving a filtering problem.

 \end{abstract}

\medskip\noindent  {\bf 2010  AMS-classification}: 60H10; 60H30; 60J60; 65C05; 65C35;
% 68U20; 
35K58.
 
 % % % % % % % % % % % % % % % % % % % % % % % % % % % % % % %
 
\section{Introduction and motivations}

\subsection{General considerations}

The idea of the present article is to focus on models which 
have a double macroscopic-microscopic face in the form
of {\it perturbation} of a so called Fokker-Planck type equation
that we call {\it generalized} Fokker-Planck equation. Our ambition is driven by two main reasons.
\begin{enumerate}
\item A {\it modeling} reason: the idea is to observe both from
a macroscopic-microscopic point of view phenomena arising from  physics, biology,
chemistry or complex systems.
\item A {\it numerical simulation} reason: to provide Monte-Carlo 
suitable algorithms to approach PDEs.
\end{enumerate}
The {\it target macroscopic} Fokker-Planck equation is 
\begin{equation}
\label{epdeIntro0}
\left \{
\begin{array}{lll}
\partial_t u &=& \frac{1}{2} \displaystyle{\sum_{i,j=1}^d} \partial_{ij}^2 \left( (\sigma \sigma^\top)_{i,j}(t,x,u) u \right) - div \left( b(t,x,u, \nabla u) u \right)\\
 &&+ \Lambda(t,x,u,\nabla u) u\ , \quad \textrm{for}\  t\in ]0,T]\ ,\\
u(0, \cdot) &=& {\bf u_0},
\end{array}
\right .
\end{equation}
where ${\bf u_0}$ is a Borel probability measure
 $\sigma: [0,T] \times \R^d \times \R \rightarrow  M_{d,p}(\R)$,
$b: [0,T] \times \R^d \times \R \rightarrow  \R^d$,
$\Lambda: [0,T] \times \R^d \times \R \times \R^d \rightarrow \R$ and $\nabla$  denotes the gradient operator. 
The initial condition in \eqref{epdeIntro0} 
means that for every continuous bounded real function $\varphi$
we have
$\int  \varphi(x)u(t,x) dx \rightarrow \int \varphi(x) {\bf u_0}(dx)$
when $t \rightarrow 0$.
When ${\bf u_0}$ admits a density, we denote it by $u_0$.
The unknown function $u:]0,T] \times \R^d \rightarrow \R$
is supposed to run in  $L^1(\R^d)$ considered as a subset
of the space of finite Radon measures $\shm(\R^d)$.
The idea consists in finding a probabilistic representation via the solution of a {\it Stochastic Differential Equation (SDE)}
whose coefficients do not depend only on time and the position
of the {\it particle} but also on its probability law. 
The {\it target  microscopic equation}  we have in mind is
\begin{equation}
\label{eq:MckeanExtIntro}
\left\{
\begin{array}{l}
Y_t=Y_0+\int_0^t \sigma\Big (s,Y_s,u(s,Y_s)\Big )dW_s+\int_0^t b\Big (s,Y_s,u(s,Y_s)\Big )ds\\ 
Y_0\,\sim\, {\bf u_0}\\ 
{\displaystyle \int \varphi(x)u(t,x)dx=\E\left [\,\varphi(Y_t)\,\,\exp\Big \{\int_0^t\Lambda\big (s,Y_s,{u}(s,Y_s),\nabla{u}(s,Y_s)\big )ds\Big \}\,\right]}\ ,\quad \textrm{for}\ t\in ]0,T]\ ,
\end{array}
\right . 
\end{equation}
for any continuous bounded real valued test function $\varphi$. 
Sometimes we denominate the third line equation of
\eqref{eq:MckeanExtIntro} the {\it linking equation}. 
When $\Lambda=0$, in equation~\eqref{eq:MckeanExtIntro},  the linking
equation simply says that $u(t,\cdot)$ 
coincides with the density of the marginal distribution $\mathcal{L}(Y_t)$. 
In this specific case, equation~\eqref{eq:MckeanExtIntro} reduces to a  {\it McKean Stochastic Differential Equation (MSDE)}, which is in general an SDE whose coefficients, at time $t$,  depend, not only on $(t,Y_t),$  but also on the marginal law $\mathcal{L}(Y_t)$. With more general functions $\Lambda$,
the role of the linking equation is more intricate since the whole history of the process $(Y_s)_{0\leq s\leq t}$ is involved. This fairly general type of equations will be called {\it McKean Feynman-Kac Equation (MFKE)} to emphasize the fact that $u(t,x)dx$ now corresponds to a
 non-conservative Feynman-Kac measure. 

An interesting feature of MSDEs (so when $\Lambda=0$) is that the law of the  process $Y$ can often be characterized as the limiting empirical distribution of a large number  of interacting particles, whose dynamics are described by a coupled system of classical SDEs. When the number of particles grows to infinity, the particles behave closely to a system of independent copies of
$Y$. This constitutes the so called {\it propagation of chaos} phenomenon, already observed in the literature when the drift and diffusion coefficients
 are Lipschitz dependent on the solution marginal law,
with respect to the Wasserstein metric, see e.g. \cite{kac, Mckeana,  Mckean, sznitman, MeleaRoel}.  Propagation of chaos is a common phenomenon arising in many
 physical contexts, see for instance \cite{ugolini} concerning
Nelson stochastic mechanics.

When $\Lambda=0$,  equation~\eqref{epdeIntro0} is a {\it non-linear} Fokker-Planck equation, it is conservative and it is known that, under mild assumptions, it describes the dynamics of the marginal probability densities, $u(t,\cdot)$, of the process $Y$. This correspondence between PDE~\eqref{epdeIntro0} with MSDE~\eqref{eq:MckeanExtIntro} and interacting particles 
%between the marginal laws of a diffusion process and a Fokker-Planck type PDE, constitutes a representation property according to which a stochastic object characterizes a deterministic one and vice versa. Such representation results 
have extensive interesting applications. 
In physics, biology
or economics, it is a way to relate a microscopic model involving interacting particles to a macroscopic model involving the dynamics of the underlying density.
 Numerically, this correspondence motivates Monte-Carlo approximation schemes for PDEs. 
In particular, \cite{bossytalay1} has  contributed to develop stochastic
 particle methods in the
 spirit of McKean to provide original numerical schemes approaching a PDE
 related to 
Burgers equation providing also the rate of convergence.

Below we list some situations of particular interest where such correspondence holds. 

\subsection{Some motivating examples}

{\bf Burgers equation}

\medskip

We fix 
 $d = p = 1$ and let $\nu>0$ and $u_0$ be a probability density on $\R$.
  We consider two equivalent specific cases of~\eqref{epdeIntro0}. 
%\begin{itemize}
%\item
 The first
 $\sigma \equiv \nu,  \ b \equiv 0, \ \Lambda(t,x,u,z) =  z$. 
The second $\sigma  \equiv \nu, \ b(t,x,u) = \frac{u}{2},    \Lambda = 0$. 
%Burgers' equation is a conservative equation.
Both instantiations correspond to the the {\it viscid Burgers equation} in dimension $d = 1$, given by
\begin{equation}
\label{eq:Burgers}
\left \{ 
\begin{array}{l}
\partial_t u = \frac{\nu^2}{2} \partial_{xx} u - u \partial_x u, \quad (t,x) \in [0,T] \times \R,  \\
u(0,\cdot) = {\bf u_0} \ .
\end{array}
\right .
\end{equation}

\noindent {\bf Generalized Burgers-Huxley equation}

\medskip

We fix 
 $d = p = 1$ and let $\nu>0$ and $u_0$ be a probability density on $\R$. We consider the particular cases
 of~\eqref{epdeIntro0} where 
 $\sigma \equiv \nu,  \ b(t,x,u)=\alpha\frac{u^n}{n+1}, \ \Lambda(t,x,u) =  \beta (1- u^n)(u^n - \gamma)$, with fixed  reals $\alpha,\beta,\gamma$ and a non-negative integer $n$. 
%
%\item{\it Generalized Burgers-Huxley equations}
This instantiation corresponds to a natural extension of Burgers equation called
{\it Generalized Burgers-Huxley equation} or {\it Burgers-Fisher equation} which is of great importance to represent non-linear phenomena in various fields such as biology~\cite{aronsonb,murray}, physiology~\cite{keener} and physics~\cite{wang}. 
 These equations have the particular interest to describe the interaction between the reaction mechanisms, convection effect, and diffusion transport.
Those are non-linear and non-conservative PDEs of the form 
\begin{equation}
\label{eq:BurgersHuxley}
\left\{
\begin{array}{l}
\partial_t u =  \frac{\nu}{2}\partial_{xx} u - \alpha u^n \partial_x u
+ \beta u (1- u^n)(u^n - \gamma),
\quad (t,x) \in [0,T] \times \R,  \\
u(0,\cdot) = {\bf u_0.} 
\end{array}
\right .
\end{equation}

%\item Fokker-Planck equation related to time reversal.

\noindent {\bf Fokker-Planck equation with terminal condition}

\medskip
\label{EDPTermCond}
The present example does not properly integrate the framework of~\eqref{epdeIntro0}. In terms of application, we are interested  by 
inverse problems that can be formulated by a PDE with terminal condition
\begin{equation} \label{EDPTerm}
\left \{
\begin{array}{lll}
\partial_t u &=& \frac{1}{2} 
\displaystyle{\sum_{i,j=1}^d} \partial_{ij}^2 \left( (\sigma \sigma^t)_{i,j}(t,x) u \right) - div \left( b(t,x) u \right)\\
 &&+ \Lambda(t,x) u\ , \quad \textrm{for}\  t\in ]0,T[\ ,\\
u(T, \cdot) &=& {\bf u_T,}
\end{array}
\right .
\end{equation}
where ${\bf u_T}$ is a prescribed probability measure.
Solving that equation by analytical means constitutes a delicate task. A probabilistic representation may 
help for studying well-posedness or providing 
numerical schemes.
% $$ \partial_t u = \sum_{i,j= 1}^d \partial^2_{ij}
% ((a a^*)^{i,j}(x) u - div(b(x) u) - \frac{div(a^* a u_t)}(u_t}).$$

%The next example is described in the section below.

%\subsubsection{Generalized porous media type equations}

Backward simulation of diffusions is a subject of active research in various domains of physical sciences and engineering, as heat 
conduction~\cite{beck1985inverse},
 material science~\cite{ renardy1987mathematical} or hydrology~\cite{bagtzoglou2003marching}.
 In particular, \textit{hydraulic inversion} is interested in inverting a diffusion phenomenon representing the concentration of a pollutant to identify the pollution source location when the final concentration profile is observed.   
The problem is in general ill-posed because either the solution is not unique 
or the solution is not stable. For this type of problem, the existence is ensured by the fact that the  observed contaminant has necessarily originated from some place at a given time (as soon as the model is correct). 
To correct the lack of well-posedness 
two regularization procedures have been proposed in the literature:
the first one relies on the notion of quasi-solution, introduced by Tikhonov
\cite{tikhonov1977solutions},
the second one on the method of quasi-reversibility, 
introduced by Lattes and Lions, \cite{lattes1969method}.
% %
Besides well-posedness, a second crucial issue consists in providing a numerical
 approximating scheme to the  backward diffusion equation. 
%when the problem is well-approximated by a regularized problem consists in providing a numerical approximating scheme to the backward diffusion equation.
%%% FIN PAS CLAIR JE PROPOSE QUELQUECHOSE
% For  the well-posedness,
% two  regularization procedures were proposed:
% the first one relies on the notion of quasi-solution, introduced by Tikhonov
% \cite{tikhonov1977solutions}  and 
% the second one on the method of quasi-reversibility, 
% introduced by Lattes and Lions, \cite{lattes1969method}.
% %
% Concerning the second issue, when the problem is well-approximated by a regularized problem consists in providing
% a numerical scheme to the backward diffusion equation.
A probabilistic representation of \eqref{EDPTerm} via
the time-reversal of a diffusion could show those issues under a new light.

\medskip
%\subsubsection
\noindent{\bf The stochastic Fokker-Planck with multiplicative noise}

\label{SSOC}

We fix $p = d,  \quad \sigma(t,x,u) = \Phi(u) Id_d\,$, where $  \Phi: \R \rightarrow \R$ and $b= \Lambda \equiv 0.$
Typical examples are the case of classical
 porous media type equation (resp. fast diffusion equation), 
when $ \Phi(u) = u ^q, 1 \le q$ (resp. $ 0 < q < 1$). 
The (singular) case 
 $\Phi(u) =  \gamma H(u-e_c)$, $H$ being the Heavisise function and $e_c$ a given threshold in $\R,$
appears in the science of complex systems, more precisely in the so called
{\it self-organized criticality}, see e.g. \cite{Bak, Caf_Loreto, BTW}. 

\begin{equation}
\label{E1.1}
\left \{
\begin{array}{lll}
\partial_t u&=& \frac{\gamma}{2} \Delta (H(u-e_c)u )
\\ 
u(0,\cdot)& =& {\bf u_0}.
\end{array}
\right.
\end{equation}
%This point will be partially discussed in Section \ref{random},
%%\ref{SStoch},
%in relation with random environment.
%%and filtering problems.
%%%\end{enumerate}
%In this case we enlarge the framework to the stochastic setting
%allowing $\Lambda$ to be a random coefficient.
%In particular, we suppose
%$\Lambda(t,x,u;\omega) = \xi(t,x)$ where $\xi$
%is a space-time (white in time and coloured in space) Gaussian 
%noise. 
%
%Coming back to Section \ref{SSOC},
The phenomenon of {\it self-organized criticality}
often is described in  two scale phases: a fast dynamics (of {\it avalanch type})
 described by the PDE
\eqref{E1.1} and a slower motion of {\it sand storming}
 modeled by the addition of a supplementary stochastic noise
$\Lambda(t,x; \omega)$.
%The complex one mixing the two scales,
% in which 
%appears a noise perturbation, see e.g. 
%\cite{BBDR},
%corresponding to a slower phenomenon of {\it sand storming}
%ding to a slower phenomenon of {\it sand storming}
In that case the {\it target macroscopic} equation is
\begin{equation}
\label{E1.1bis}
\left \{
\begin{array}{lll}
\partial_t u&=& \frac{\gamma}{2} \Delta (H(u-e_c)u ) + \Lambda(t,x;\omega) u
\\ 
u(0,\cdot)& =& {\bf u_0},
\end{array}
\right.
\end{equation}
where $\Lambda(t,x;\omega)$ is a quenched realization of a space-time coloured (ideally
white) noise. 
%\item{\bf Burgers equation}.
%
The SPDE will be represented by a MSDE 
 in random environment, see Section \ref{random}. 

 % In a first step one sets $\xi \equiv 0$, in which one
 % concentrates on the {\it fast} phase.

%Organization of the paper
%--------------------------------------------------
\subsection{Structure of the paper}

In the rest of the paper, to simplify notations, most of the results are stated in the one-dimensional setting. The generalization to the multi-dimensional case is straightforward. 

The paper is organized as follows. Next section presents a brief review of basic situations where Fokker-Planck equations can be represented by  MSDEs which in turn can be represented by interacting particles systems. Section~\ref{sec:NonCons}, considers the case of generalized Fokker-Planck equations in the sense of~\eqref{epdeIntro0} with a non-zero term $\Lambda$ allowing to take into account  non-conservative PDEs including a large class  of semi-linear PDEs. Section~\ref{sec:jumps} highlights the correspondence between MFKEs and MSDEs with jumps which paves the way to a great variety of numerical approximations schemes for non-linear PDEs. Section~\ref{TimeRev} is devoted to a particular inverse problem which consists in modeling backwardly in time the evolution of a Fokker-Planck equation with a given terminal condition. This problem can be related to a time-reversed SDE which in turn can be represented by a MSDE. 
In Section~\ref{random} we analyze the well-posedness of generalized Fokker-Planck equation where the term $\Lambda$ in~\eqref{epdeIntro0} may involve an exogenous noise resulting in a Stochastic non-linear PDE. Finally, in Section~\ref{sec:control}, we consider a stochastic control problem for which the associated  Hamilton-Jacobi-Bellman equation can be represented by a MFKE. 
%------------------------------------------------------------------------
%In relation with example 
%this consitutes the first motivation for 
%a non-conservative (this time stochastic) PDE.
%%% COGHI, GESS

% \cite{BarbuRockFPE},

% REVOIR ET CORRIGER
 %AJOUTER REFERENCES A BARBU ROECKNER TREVISAN FIGALLI

%%%% QUOTE \cite{BarbuRockFPE} \cite{BarbuRockSIAM}

 \section{McKean representations of non linear Fokker-Planck equations}
 \label{S4}
%-----------------------------------------------------------------

\setcounter{equation}{0}

In this section, we recall some standard situations where a Fokker-Planck PDE can be represented by an SDE which in turn can be approached by an interacting particles system.

\subsection{Probabilistic representation of linear Fokker-Planck equations}
%----------------------------------------------------------------
Suppose there exists a  solution  $(Y_t)_{t\in [0,T]}$ (in law)  to the SDE
\begin{equation}
\label{eq:Diff}
\left\{
\begin{array}{l}
{\displaystyle Y_t=Y_0+\int_0^t \sigma(s,Y_s)dW_s+\int_0^t b(s,Y_s)ds}, t \in [0,T],\\ 
Y_0\,\sim\,{\bf u_0}\ ,
\end{array}
\right . 
\end{equation}
where $W$ is a real valued Brownian motion on $[0,T]$ and $\bf u_0$ is a probability measure on $\R$. 
A direct application of It\^o formula shows that the marginal probability laws   $(\mu(t,\cdot):=\mathcal{L}(Y_t))_{t\in [0,T]}$ generate
a distributional solution of the  linear Fokker-Planck PDE
%%% NADIA. PROBLEME MESURABILITE
\begin{equation}
\label{eq:PDEFK}
\left\{
\begin{array}{l}
{\displaystyle \partial_t \mu =\frac{1}{2}\partial_{xx}^2(\sigma^2(t,x)\mu)-\partial_x(b(t,x)\mu)}\\ 
\mu(0,dx)={\bf u_0}(dx).
\end{array}
\right .
\end{equation}
This naturally suggests a Monte Carlo algorithm to approximate the above linear PDE, consisting in simulating $N$ i.i.d. particles $(\xi^{i})_{i=1,\cdots N}$  with $N$ i.i.d. Brownian motions $ (W^i)_{i=1,\cdots N}$ i.e. 
\begin{equation}
\label{eq:MC}
\left\{
\begin{array}{l}
{\displaystyle \xi^{i}_t=\xi^{i}_0+\int_0^t \sigma(s,\xi^{i}_s)dW^i_s+\int_0^t b(s,\xi^{i}_s)ds}\\
\xi^i_0\quad \textrm{i.i.d.}\,\sim\,{\bf u_0}\\
{\displaystyle \mu^N_t=\frac{1}{N}{\displaystyle \sum_{j=1}^N \delta_{\xi^{j}_t}}}.
\end{array}
\right . 
\end{equation}
Then the law of large numbers provides the convergence of the empirical approximation 
$\mu^N_t\xrightarrow[N\rightarrow \infty] {}\mu(t,\cdot)$,  the solution of the Fokker-Planck equation~\eqref{eq:PDEFK}.

\subsection{McKean probabilistic representation of non-linear Fokker-Planck equation}

\label{S2.2}

%-------------------------------------------------------------
%McKean representation of PDEs:  regularized nonlinear setting \textit{In the spirit of \cite{Mckean67}}
We consider the non-linear SDE in the sense of
 McKean (MSDE) 
%~\cite{Mckean}
\begin{equation}
\label{eq:MckeanSDE}
\left\{
\begin{array}{l}
{\displaystyle Y_t=Y_0+\int_0^t \sigma\Big (s,Y_s,(K\ast \mu)(s,Y_s)\Big )dW_s+\int_0^t b\Big (s,Y_s,(K\ast \mu)(s,Y_s)\Big )ds}\\ 
Y_0\,\sim\,{\bf u_0}\\ 
\mu(t,\cdot)\ \textrm{is the probability law of}\ Y_t\ ,t \in [0,T],
\end{array}
\right . 
\end{equation}
 whose solution is a couple $(Y,\mu)$.  
Here $\sigma, b$ are Lipschitz,
$K \,:\,\R\times\R\rightarrow\R$ denotes a Lipschitz continuous convolution kernel such that $(K\ast\mu)(t,y):=\int K(y,z)\mu(t,dz)$ for any $y\in\R$. We emphasize that this type of regularized dependence of the drift and diffusion coefficients on $\mu$ is essentially different (and in general easier to handle) from
a pointwise dependence where the coefficients $b$ or $\sigma$ may depend on the value of the marginal density at the current particle position $\frac{d\mu}{dx}(s,Y_s)$. 
This regularized or non-local dependence on the time-marginals  $\mu (t,\cdot)$ is a particular
case of the framework when the diffusion and drift coefficients are Lipschitz  with respect to $\mu (t,\cdot)$
according to the the Wasserstein metric.

Again, by It\^o formula,
 given a solution $(Y,\mu)$ of~\eqref{eq:MckeanSDE}, $\mu$ 
 solves the non-local non-linear PDE
\begin{equation}
\label{eq:PDEMckean}
\left\{
\begin{array}{l}
{\displaystyle \partial_t \mu =\frac{1}{2}\partial_{xx}^2\Big (\sigma^2 (t,x,K\ast \mu)\mu\Big )-\partial_x\Big (b(t,x,K\ast\mu)\mu\Big )}\\
\mu(0,dx)={\bf u_0}(dx),
\end{array}
\right .
\end{equation}
in the sense of distributions.
In this setting, the well-posedness of~\eqref{eq:MckeanSDE}
 relies on a fixed point argument in the 
space of trajectories under the Wasserstein metric, see e.g. 
%Developing that idea, \cite{Funaki84} proved 
%the existence of a solution to SDE \eqref{eq:MckeanSDE}, with vanishing drift. 
%That approach was generalized in~
\cite{sznitman}, at least in the case when the diffusion
term does not depend on the law.  
We will denominate this situation as the {\it traditional} setting.
% It is significant to mention that if the dependence
% of diffusion and drift coefficients is Wasserstein continuous, 
% the existence (and under further restrictions uniqueness) 
% was established for solutions in law.

%
Deriving a Monte-Carlo approximation scheme from this probabilistic representation already becomes more tricky since it can no more rely on independent particles but should involve an interacting  particles system as initially proposed in~\cite{kac,sznitman}.  Consider $N$ interacting particles $(\xi^{i,N})_{i=1,\cdots N}$  with $N$ i.i.d. Brownian motions $ (W^i),$ i.e. 
\begin{equation}
\label{eq:MckeanPart1}
\hspace{-1cm}
\left\{
\begin{array}{l}
{\displaystyle \xi^{i,N}_t=\xi^{i,N}_0+\int_0^t \sigma\Big (s,\xi^{i,N}_s, (K\ast \mu^N_s)(\xi^{i,N}_s)\Big )dW^i_s+\int_0^t b\Big (s,\xi^{i,N}_s,(K\ast \mu^N_s)(\xi^{i,N}_s)\Big )ds}\\
\xi^{i,N}_{0}\ \textrm{i.i.d.}\ \sim\, {\bf u_0}\\
{\displaystyle \mu^N_t=\frac{1}{N}{\displaystyle \sum_{j=1}^N \delta_{\xi^{j,N}_t}}},
\end{array}
\right . 
\end{equation}
with $(K\ast\mu^N_t)(y)=\frac{1}{N}\sum_{j=1}^N K(y,\xi^{j,N}_t)\ .
$
The above system defines a so-called \textit{weakly interacting} particles system, as  pointed out in~\cite{Oelschlager84}. This terminology underlines the fact that any particle interacts with the rest of the population with a vanishing impact of order $1/N$.
In this setting, at least when the 
diffusion coefficient does not depend ton the law, \cite{sznitman} proves the so called {\it chaos propagation}
 which means that
  $(\xi^{i,N}_t)_{i=1,\cdots N}$ asymptotically behaves as an i.i.d. sample according to $\mu (t,\cdot)$ as the number of particles $N$ grows to infinity, where $\mu$ is the solution of the regularized non-linear PDE~\eqref{eq:PDEMckean}. 
This in particular 
 implies the convergence of the empirical measures 
$\mu^N_t\xrightarrow[N\rightarrow \infty] {}\mu(t,\cdot)$ with the rate $C/\sqrt{N}$ inherited from the law of large numbers. 
%

% pointwise nonlinear setting~\cite{JourdainEtMeleard98}
As already announced, the case where the coefficients depend pointwisely on
the density law $u(t,\cdot)$ of $\mu(t,\cdot)$,  $t > 0$, 
 is far more singular. Indeed  the dependence of the coefficients on the  law of $Y$ is no  more continuous with respect to the Wasserstein metric. In this context, well-posedness results rely generally on analytical methods.
 One important contribution in this direction is reported in~\cite{JourMeleard}, where strong existence and pathwise uniqueness are established when  
the diffusion coefficient $\sigma$  and the drift $b$ exhibit pointwise dependence on $u$ but are assumed to satisfy strong smoothness assumptions
together with the initial condition.
In this case, the solution $u$ is a classical solution of the PDE
\begin{equation}
\label{eq:PDEMckeanb}
\left\{
\begin{array}{l}
{\displaystyle \partial_t u =\frac{1}{2}\partial_{xx}^2\Big (\sigma^2 (t,x,u(t,x))u\Big )-\partial_x\Big (b(t,x,u(t,x))u\Big )} \\
u(0,x)={\bf u_0}(dx),
\end{array}
\right .
\end{equation}
which is formally derived from \eqref{eq:PDEMckean} setting
$K(x,y) = \delta_{0}(x-y)$.
Let us fix $K^\varepsilon$ being a  mollifier (depending on a window-width
parameter $\varepsilon$),  such that 
$ {\displaystyle  
 K^{\varepsilon}(x,y)=\frac{1}{\varepsilon^d}\phi(\frac{x-y}{\varepsilon})\xrightarrow[\varepsilon \rightarrow 0] {} \delta_0(x-dy)\ .
}$
As in \eqref{eq:MckeanPart1}, we consider
the $N$ interacting particles $(\xi^{i,N})_{i=1,\cdots N}$ 
solving
\begin{equation}
\label{eq:MckeanPart2}
\left\{
\begin{array}{l}
{\displaystyle \xi^{i,N}_t=\xi^{i,N}_0+\int_0^t \sigma\Big (s,\xi^{i,N}_s, u^{N,\varepsilon}_s(\xi^{i,N}_s)\Big )dW^i_s+\int_0^t b\Big (s,\xi^{i,N}_s,u^{N,\varepsilon}_s(\xi^{i,N}_s)\Big)ds}\\
\xi^{i,N}_0\ \textrm{i.i.d.}\,\sim\,{\bf u}_0\\
{\displaystyle u^{N,\varepsilon}_t=\frac{1}{N}{\displaystyle \sum_{j=1}^N K^\varepsilon}(\cdot,\xi^{j,N}_t)}. 
%%%% ,= N^epsilon ^d
 \end{array}
\right . 
\end{equation}
%
%According to ~\cite{Oelschlager84}, 
%the system~\eqref{eq:MckeanPart2} defines a so-called \textit{moderately interacting} particle system
Under the smooth assumptions on $b,\sigma$, $u_0$ mentioned before
and non-degeneracy of $\sigma$, 
\cite{JourMeleard} proved the convergence of the regularized particle approximation $u^{N,\varepsilon}_t$ to the solution $u$ of
 the pointwise non-linear PDE~\eqref{eq:PDEMckeanb}
as soon as 
$\varepsilon (N)\xrightarrow[N\rightarrow \infty]{} 0$ slowly enough.
According to ~\cite{Oelschlager84}, 
the system~\eqref{eq:MckeanPart2} defines a so-called \textit{moderately interacting} particle system
with
 $ u^{N,\varepsilon}_t(x) =
\frac{1}{N \varepsilon^d}  \sum_{j=1}^N \phi(\frac{x - \xi^{j,N}_t}{\varepsilon}).$ 
Indeed  as the window width of the kernel, $\varepsilon$, goes to zero, 
the number of particles that significantly impact a single one is of order ${N \varepsilon^d}$ with a strength of interaction of order  $ \frac{1}{N \varepsilon^d}$. 
In contrast, when $\varepsilon$ is fixed, we recover the weakly interacting situation in which case the strength of interaction of each particle is of order
$\frac{1}{N}$ which is smaller than  $ \frac{1}{N \varepsilon^d}$.
%in the second case one speaks of \textit{weakly interacting}   particle system.
%In the first case the dynamics of each particle is impacted by less particles than in the second one.
In this case of moderate interaction, the propagation of chaos occurs with a slower rate
than $C/\sqrt{N}$ and depends exponentially on the space dimension.
 \cite{JourMeleard} constitutes an extension of the weak propagation of chaos of moderately interacting particles proved in~\cite{Oelschlager84}
 for the limited case of identity diffusion matrix.  

The peculiar case where the drift vanishes and the diffusion coefficient
$\sigma(u(t,Y_t))$ has a pointwise dependence on the law density $u(t,\cdot)$ of $Y_t$
  has been more particularly studied in~\cite{Ben_Vallois} for classical 
porous media type equations and
\cite{BRR1,BRR2,BCR2,BCR3,BarbuRockSIAM} who obtain well-posedness results for 
measurable and possibly singular  functions $\sigma$.
In that case the solution  $u$ of the associated PDE \eqref{epdeIntro0},
is understood in the sense of distributions.

%-----------------------------------------------------------
\section{McKean Feynman-Kac 
 representations for  non-conservative and non-linear PDEs}
\label{sec:NonCons}

\setcounter{equation}{0}

%-----------------------------------------------------------
%%% SPECIFIER DES HYPOTHESES A METTRE

The idea of generalizing MSDEs to MFKEs~\eqref{eq:MckeanExtIntro}
 was originally introduced in the sequence of papers~\cite{LOR1,LOR2,LOR4}, 
with an earlier contribution in \cite{BRR3}, where
$\Lambda(t,x,u, \nabla u) = \xi_t(x),$ $\xi$ being the sample
of a Gaussian noise random field, white in time and regular in space,
see Section \ref{random}.
The goal was to provide some probabilistic representation for non-conservative non-linear PDEs~\eqref{epdeIntro0} by introducing some exponential weights defining Feynman-Kac measures 
instead of probability measures. An interesting aspect of this strategy is that it is potentially able to represent an extended class of second order non-linear PDEs.
%\cite{LOR3, LieberOR} studied with two different techniques the probabilistic representation of \eqref{epdeIntro0} in the semilinear case using the techniques of mild solutions.
One particularity of MFKE equations is that the probabilistic
representation involves the past of the process (via the exponential weights). In this context, it is worth
to quote the recent paper \cite{tomasevic} which
proposes a probabilistic representation, which also includes a dependence on the past, in relation with Keller-Segel model with application to chemiotaxis.

It is important to consider carefully the two major 
features differentiating the MFKE~\eqref{eq:MckeanExtIntro}
from the traditional setting of MSDEs. 
To recover the traditional setting one has to do the following.
\begin{enumerate}
\item
  First, one has to put $\Lambda=0$ in the third line equation of~\eqref{eq:MckeanExtIntro}
  Then $u(t,\cdot)$ is explicitly given by the third line equation 
  of ~\eqref{eq:MckeanExtIntro} and reduces to the density of 
the marginal distribution, $\mathcal{L}(Y_t)$. 
When $\Lambda \neq 0$, the relation between $u(t,\cdot)$ and the process
$Y$ is more complex.  
Indeed, not only does $\Lambda$ embed an additional non-linearity with respect to $u$, but it also involves the whole past trajectory $(Y_s)_{0\leq s\leq t}$ of the process $Y$.  
\item Secondly, one has to replace
  the pointwise dependence $b(s,Y_s,u(s,Y_s))$ in equation
~\eqref{eq:MckeanExtIntro}
with a mollified dependence $b(s,Y_s,\int_{\R^d} K(Y_s- y)u(s,y)dy)$, where the dependence with respect to $u(s, \cdot)$  is Wasserstein  continuous.
Here $K: \R \rightarrow \R$ is a convolution kernel.
\end{enumerate}

One interesting aspect of probabilistic representation~\eqref{eq:MckeanExtIntro} is that it  naturally yields numerical approximation schemes involving weighted interacting particle systems.
More precisely, we consider $N$ interacting particles $(\xi^{i,N})_{i=1,\cdots N}$  with $N$ i.i.d. Brownian motions $ (W^i)_{i=1,\cdots N}$, i.e. 
\begin{equation}
\label{eq:Mckean}
\left\{
\begin{array}{l}
\xi^{i,N}_t=\xi^{i,N}_0+\int_0^t \sigma\Big (s,\xi^{i,N}_s, u^{N,\varepsilon}_s(\xi^{i,N}_s)\Big )dW^i_s+\int_0^t b\Big (s,\xi^{i,N}_s,u^{N,\varepsilon}_s(\xi^{i,N}_s)\Big )ds\\
\xi^{i,N}_0\ \textrm{i.i.d.}\,\sim\,{\bf u_0}\\
u^{N,\varepsilon}_t(\xi^i_t)={\displaystyle \sum_{j=1}^N \,\omega_t^{j,N}\,K^{\varepsilon}(\xi^{i,N}_t-\xi^{j,N}_t)}\ ,
\end{array}
\right . 
\end{equation}
where the mollifier $K^{\varepsilon}$ is such that 
$K^{\varepsilon}(x)=\frac{1}{\varepsilon^d}\phi(\frac{x}{\varepsilon})\xrightarrow[\varepsilon \rightarrow 0] {} \delta_0$ and the weights $\omega_t^{j,N}$ for  $j=1,\cdots, N$  verify
\begin{eqnarray*}
\omega_t^{j,N}&:=&\exp\left \{\,\int_0^t \Lambda \Big (r,\xi_r^{j,N},u_r^{\varepsilon , N}(\xi_r^{j,N}),\nabla u_r^{\varepsilon , N}(\xi_r^{j,N})\Big )\,dr\right \}\\
&=&
\omega^{j,N}_{s}\exp\left \{\,\int_{s}^{t} \Lambda \Big (r,\xi_r^{j,N},u_r^{\varepsilon , N}(\xi_r^{j,N}),\nabla u_r^{\varepsilon , N}(\xi_r^{j,N})\Big )\,dr\right \}.
\end{eqnarray*}

%at least when $\Phi$ depends on $u$. 
%However, we believe that (MFKE) also have their own interest, independently from the related PDE. 
%Coming back to \eqref{eq:PDE_original},
  
\cite{LOR3,LOR4} consider the case of pointwise semilinear PDEs of the form
\begin{equation}
\label{eq:PDEMckeanExtb}
\left\{
\begin{array}{l}
\partial_t u =\frac{1}{2}\partial_{xx}^2(\sigma^2(t,x)u)-\partial_x b(t,x)u) 
+\Lambda(t,x,u,\nabla u)u \\
u(0,x)=u_0(x),
\end{array}
\right .
\end{equation}
for which the target probabilistic representation is 
\begin{equation}
\label{eq:MckeanExt1}
\left\{
\begin{array}{l}
Y_t=Y_0+\int_0^t \sigma\Big (s,Y_s\Big )dW_s+\int_0^t b\Big (s,Y_s\Big )ds\\ 
Y_0\,\sim\,{\bf u_0}\\ 
{\displaystyle \int \varphi(x)u_t(x)dx:=\E\left [\,\varphi(Y_t)\,\,\exp\Big \{\int_0^t\Lambda\big (s,Y_s,{u}_s(Y_s),\nabla u_s(Y_s)\big )ds\Big \}\,\right]}.
\end{array}
\right . 
\end{equation}
We set 
\begin{equation} \label{ELt}
L_t f := \halb \sigma^2(t,x) f''(x) + b(t,x) f'(x), t \in ]0,T[, \quad \textrm{for any}\ f \in C^2(\R).
\end{equation}
Let us consider the family of Markov 
transition functions $P(s,x_0,t, \cdot)$
associated with $(L_t)$, see \cite{LOR3}.
We recall that if $X$ is a processs solving the first line of \eqref{eq:Mckean}
with $X_s \equiv x_0 \in \R$, then
$ \int_\R P(s,x_0,t,x) f(x) dx = \E(f(X_t)), t \ge s, $
for every bounded Borel function $f :\R \rightarrow \R$. 
$u: [0,T] \times \R \rightarrow \R$ will be called  {\bf{mild solution}} of  
\eqref{eq:PDEMckeanExtb} (related to $(L_t)$)
if for all $\varphi \in \shc_0^{\infty}(\R)$, $t \in [0,T]$,
\begin{eqnarray}
\label{eq:DefMildSol}
\int_{\R^d} \varphi(x) u(t,x)dx & = & \int_{\R^d} \varphi(x) \int_{\R^d} 
{\bf u_0}(dx_0) P(0,x_0,t,dx)  \nonumber \\
&& + \; \int_{[0,t] \times \R^d} \Big( \int_{\R^d} \varphi(x) P(s,x_0,t,dx) \Big) \Lambda(s,x_0,u(s,x_0), \nabla u(s,x_0)) u(s,x_0)dx_0 ds.   \nonumber 
\end{eqnarray}
The following theorem states conditions ensuring equivalence between~\eqref{eq:MckeanExt1} and~\eqref{eq:PDEMckeanExtb}  together with the convergence of the related particle approximation~\eqref{eq:Mckean}. 
\begin{thm}
\label{theo:A1} We suppose that $\sigma$ and $b$ are Lipschitz with linear growth
 and $\Lambda$ is bounded.
%  and non-degeneracy
% of $\sigma$, the following results are established.
\begin{enumerate}
\item Let $u:[0,T] \times \R \rightarrow \R \in L^1 ([0,T]; W^{1,1}(\R^d)$.
$u$ is a mild solution of  PDE~\eqref{eq:PDEMckeanExtb}
if and only if $u$ verifies \eqref{eq:MckeanExt1}. 
% with the probabilistic representation~\eqref{eq:MckeanExt1};
\item  Suppose that $\sigma \ge c > 0$ and $\Lambda$ is uniformly Lipschitz 
w.r.t. to $u$ and $\nabla u$.
There is a unique mild solution in $L^1 ([0,T]; W^{1,1}(\R) \cap L^\infty ([0,T] \times \R)$
of
%Well-posedness of~
\eqref{eq:PDEMckeanExtb}, therefore also of \eqref{eq:MckeanExt1}. 
\item Under the same assumption of item 2., the particle approximation $u^{N,\varepsilon}$~\eqref{eq:Mckean}  converges
in $L^1 ([0,T]; W^{1,1}(\R)$
to the solution of~\eqref{eq:PDEMckeanExtb} as $N\rightarrow \infty$ and $\varepsilon (N)\rightarrow 0$ slowly enough.
\end{enumerate}
\end{thm}
Item 1. was the object of  
 Theorem 3.5 in \cite{LOR3}.
 Item 2. (resp. item 3.)
 was treated in Theorem 3.6 (resp. Corollary 3.5) in \cite{LOR3}.  
\begin{rem} \label{RError} The error induced by the discrete time approximation
of the particle system was evaluated in \cite{LOR4}.
\end{rem}

\cite{LieberOR} considers the case where $b$ is replaced by $b + b_1$
where $b$ is only supposed bounded Borel, without regularity assumption
on the space variable. In particular they treat the
pointwise semilinear PDEs of the form
\begin{equation}
\label{eq:PDEMckeanExtc}
\left\{
\begin{array}{l}
\partial_t u =\frac{1}{2}\partial_{xx}^2(\sigma^2(t,x)u)-\partial_x\Big 
(\big (b(t,x)+b_1(t,x,u)\big )u\Big )+\Lambda(t,x,u)u\\ 
u(0,x)=u_0(x),
\end{array}
\right .
\end{equation}
for which the target probabilistic representation is
\begin{equation}
\label{eq:MckeanExt2}
\left\{
\begin{array}{l}
Y_t=Y_0+\int_0^t \sigma (s,Y_s )dW_s+\int_0^t \big [  b (s,Y_s )+
b_1\Big (s,Y_s, u(s,Y_s)\Big )\big ]ds\\ 
Y_0\,\sim\,{\bf u_0}\\ 
{\displaystyle \int \varphi(x)u_t(x)dx:=\E\left [\,\varphi(Y_t)\,\exp\Big \{\int_0^t\Lambda\big (s,Y_s,{u}_s(Y_s)\big )ds\Big \}\,\right]}.
\end{array}
\right . 
\end{equation}
The following theorem states conditions ensuring equivalence between~\eqref{eq:MckeanExt2} and~\eqref{eq:PDEMckeanExtc}  together with well-posedness conditions for both equations. 
\begin{thm} We formulate the following assumptions.
\begin{enumerate}
\item The PDE in the sense of distributions $\partial u = L_t^* u_t$ admits
as unique solution $u \equiv 0,$
where $L_t$ was defined in \eqref{ELt}.
 \item $b$ is bounded measurable and $\sigma$ is continuous $\sigma \ge c > 0$ for some constant $c > 0$.
\item $b_1, \Lambda:[0,T]\times \mathbb{R} \times \mathbb{R} \rightarrow \mathbb{R}$ is uniformly bounded, Lipschitz with respect to the third argument.
\item The family of Markov 
transition functions 
associated with $(L_t)$, are of the form
   $P(s,x_0,t, dx) =  p(s,x_0,t,x) dx,$,
i.e. they admit measurable densities $p$.
\item The first order partial derivatives of the map $x_0 \mapsto p(s,x_0,t,x)$ exist in the distributional sense.
\item For almost all $0 \leq s < t \leq T$ and $x_0,x \in \R$ there are constants $C_u,c_u>0$ 
such that 
\begin{align}
\label{eq:1011a}
 p(s,x_0,t,x) \leq C_u
q(s,x_0,t,x)
\quad 
%\end{align}
and
\quad
%\begin{align}
%\label{eq:1011}
\left\vert \partial_{x_0} p(s,x_0,t,x) \right\vert \leq C_u
 \frac{1}{\sqrt{t-s}} q(s,x_0,t,x)\ ,
\end{align}
where
$q(s,x_0,t,x):=\left(\frac{c_u(t-s)}{\pi}\right)^{\frac{1}{2}} e^{-c_u \frac{\vert x-x_0 \vert^2}{t-s}}$ 
is a Gaussian probability density.
\end{enumerate}
The following results hold.
\begin{enumerate}
\item Let $u  \in (L^1 \cap L^\infty)([0,T] \times \R)$.
$u$ is a  solution of  PDE~\eqref{eq:PDEMckeanExtc} in the sense of distributions
if and only if
$u$ verifies  \eqref{eq:MckeanExt2} for a solution $Y$
 in the sense of probability laws. 
\item There is a unique solution $u  \in (L^1 \cap L^\infty)([0,T] \times \R)$
in the sense of distributions of PDE~\eqref{eq:PDEMckeanExtc} (and therefore of\eqref{eq:MckeanExt2}).
\end{enumerate}
\end{thm}
The result 1. (resp. result 2.) was the object of Theorem 12.
 (resp. Proposition 16., Theorems 13., 22.) of  \cite{LieberOR}.

\begin{rem} %%% STRONG SOLUTIONS
  Under more restrictive assumptions on $b$,  
item 3. of Theorem 13. in
\cite{LieberOR} 
states the well-posedness of \eqref{eq:MckeanExt2}) in the sense of strong existence
and pathwise uniqueness.

\end{rem}

\cite{LOR1} and \cite{LOR2} studied a mollified version of
\eqref{epdeIntro0}, whose probabilistic representation 
 falls into the Wasserstein continuous
  traditional setting mentioned above. Following the spirit of \cite{sznitman}, a fixed point argument was carried out  in the general case in~\cite{LOR1} to prove well-posedness of 
\begin{equation}
\label{eq:MckeanExt2Bis}
\left\{
\begin{array}{l}
Y_t=Y_0+\int_0^t \sigma\Big (s,Y_s,K\ast u_s(Y_s)\Big )dW_s+\int_0^t b\Big (s,Y_s,K\ast u_s(Y_s)\Big )ds\\ 
Y_0\,\sim\,{\bf u_0}\\ 
{\displaystyle (K\ast  u_t)(x):=\E\left [K(x-Y_t)\,\,\exp\Big \{\int_0^t\Lambda\big (s,Y_s,{K\ast u}_s(Y_s)\big )ds\Big \}\,\right]},
\end{array}
\right . 
\end{equation}
where $K: \R \rightarrow \R$ is a mollified kernel. 
% of~\eqref{eq:MckeanExt2}.
We remark that if $(Y,u)$ is a solution of \eqref{eq:MckeanExt2Bis},
then  $u$ is a solution 
(in the sense of distribution) of
\begin{equation}
\label{eq:PDEMckeanExta}
\left\{
\begin{array}{l}
\partial_t u =\frac{1}{2}\partial_{xx}^2(\sigma^2(t,x,K\ast u)u)-\partial_x(b(t,x,K\ast u)u)+\Lambda(t,x,K\ast u)u\\ 
u(0,x)=u_0(x).
\end{array}
\right .
\end{equation}

\begin{rem}
\label{theo:A2}
\begin{enumerate} %%%% \cite{LOR2}
\item Existence and uniqueness results (in the strong sense
and in the sense of probability laws) for the MFKE \eqref{eq:MckeanExt2Bis}
 are established under various technical assumptions, see \cite{LOR2}.
%with the relation $\bar u=K\ast u$. 
\item
Chaos propagation for the interacting particle system~\eqref{eq:Mckean} providing an approximation to the regularized PDE~\eqref{eq:PDEMckeanExta}, as $N\rightarrow\infty$ 
(for fixed $K$), \cite{LOR1}.
\end{enumerate}
\end{rem}

%FIN VERIFICATIONS FRANCESCO

%-------------------------------------------------------------------------------
\section{McKean representation of a Fokker-Planck equation
with terminal condition}
%-----------------------------------------------------------

\label{TimeRev}
\setcounter{equation}{0}

Let us consider the PDE with terminal condition~\eqref{EDPTerm}
and $\Lambda = 0$.
\begin{equation} \label{EDPTerm0}
\left \{
\begin{array}{lll}
\partial_t u &=& \frac{1}{2} 
\displaystyle{\sum_{i,j=1}^d} \partial_{ij}^2 \left( (\sigma \sigma^t)_{i,j}(t,x) u \right) - div \left( b(t,x) u \right) \\
%\quad \textrm{for}\  t\in ]0,T[\ ,\\
u(T,dx) &=& {\bf u_T}(dx),
\end{array}
\right .
\end{equation}
where ${\bf u_T}$ is a given Borel probability measure.
In the present section we assume the following.
\begin{ass} \label{AA}
Suppose that \eqref{EDPTerm0} admits uniqueness, i.e. that
there is at most one solution of \eqref{EDPTerm0}.
\end{ass}
\begin{rem}
 %%% FRANCESCO Ev. specifier la classe des mesures de depart. 
Different classes of sufficient conditions for that
are provided in \cite{LucasOR}.
\end{rem}

% \begin{equation*}
% \tilde{b}^i\left(s,y\right) := \left[\frac{\mathop{div}\left(\widehat{\Sigma}_{j.}\left(s,y\right)p^i_{s}\left(y\right)\right)}{p^i_{s}\left(y\right)}\right]_{j\in[\![1,d]\!]} - \widehat{b}\left(s,y\right) 
% \end{equation*}

% To avoid technicalities which may complicate the task of the reader we continue in dimension $d=1$. 
A natural representation of \eqref{EDPTerm0} 
is the following MSDE,
where $\beta$ is a Brownian motion.
\begin{equation}\label{MK}
%\forall t \in [0,T],
\left \{
\begin{array}{c}
\displaystyle Y_t = \xi - \int^{t}_{0} \tilde b\left(s,Y_s; v_s \right) ds + 
%\int^{t}_{0} 
%\left[\frac{\mathop{div}
%\left(\Sigma_{i.}\left(T- r,Y_r\right)v_{r}\left(Y_r\right)\right)}{v_{r}\left(Y_r\right)}\right]_{i\in[\![1,d]\!]}dr + 
\int^{t}_{0} \sigma \left(T-s,Y_s\right)d\beta_s, t \in [0,T]    \\
\int_{\R^d} v_{t}(x) \varphi(x) dx = \E(\varphi(Y_t)), t \in [0,T]  \\
\xi \sim {\bf u_T}, 
%\ \rm{density\ law\ of} \ Y_t,\\
\end{array}
\right.
\end{equation}
where $\tilde b(s,y;v_s) = 
(\tilde b^1(s,y;v_s) ,\ldots, \tilde b^d(s,y;v_s))$ is defined as
\begin{equation} \label{MKB}
\tilde{b}\left(s,y;v_s\right) := \left[\frac{\mathop{div}_y\left(\sigma \sigma^t_{j.}
\left(T-s,y\right)
v_{s}\left(y\right)\right)}{v_{s}\left(y\right)}\right]_{j\in[\![1,d]\!]}
 - b \left(T-s,y\right).
\end{equation}
For $d= 1$ previous expression gives
\begin{equation} \label{Ed=1}
\tilde{b}(s,y;v_s) := \frac{\left(\sigma^2(T-s,\cdot)
v_{s} \right)'}{v_{s}}(y) - b \left(T-s,y\right).
\end{equation}
\begin{rem} \label{RHP}
\eqref{MK} is in particular fulfilled if $Y$ is the time reversal process
$\hat X_t := X_{T-t}$ of a diffusion 
satisfying the SDE
\begin{equation}
\label{eq:X} 
\left \{
\begin{array}{l}
{\displaystyle 
X_t=X_0+\int_0^t b(s,X_s)ds+\sigma(s,X_s)dW_s, t \in [0,T]
}\\
X_0\sim {\bf u_0} \in\mathcal{P}(\R).
\end{array}
\right .
\end{equation}
%where $W$ denotes the standard Brownian motion on $\R$.
This happens under locally Lipschitz conditions on $\sigma$
and $b$ and minimal regularity conditions 
on the law density $p_t$ of $X_t$. 
Indeed in
\cite{haussmann_pardoux}, the authors prove that 
\begin{equation}\label{EHP}
%\forall t \in [0,T],
\displaystyle \hat X_t = X_T + \int^{t}_{0} \tilde b\left(s,\hat X_s; p_{T-s}\right)
ds + 
\int^{t}_{0} \sigma \left(T-s,\hat X_s\right)d\beta_s, \ t \in [0,T], 
\end{equation}
where $\tilde b$ is defined in \eqref{MKB} and
% replacing $v$ with 
$p_{t}$ is the density of $X_t$.
We emphasize that the main difference between 
 \eqref{MK} and \eqref{EHP}  is that in the first equation
the solution is a couple $(Y,v)$, in the second one, a solution is just $Y$,
$p$ being exogeneously defined by \eqref{eq:X}.
\end{rem}

We observe now that a solution $(Y,v)$ of \eqref{MK}
provides a solution $u$ of
\eqref{EDPTerm0}. This justifies indeed the terminology
of probabilistic representation.
\begin{prop} \label{PlinkingTR}
\begin{enumerate}
\item Let $(Y,v)$ be a solution of \eqref{MK}. Then
 $u(t, \cdot) := v(T-t,\cdot), t \in [0,T]),$
is a solution of \eqref{EDPTerm0} with terminal value ${\bf u_T}$.
\item 
If \eqref{EDPTerm0} admits at most one solution, then there is at
 most one $v$ such that
 $(Y,v)$ solves \eqref{MK}.
%\end{prop}
\end{enumerate}

\end{prop}
\begin{proof} \
 In order to avoid technicalities which complicate the task of the reader we 
 express the proof for $d = 1$.
We prove 1. since 2. is an immediate consequence of 1.

Let $\phi \in \mathcal{C}^{\infty}\left(\mathbb{R}\right)$
with compact support
 and $t\in[0,T]$. 
It\^{o} formula gives
\begin{equation*}
 \mathbb{E}\left[\phi\left(Y_{T-t}\right)\right] - \int_{\R^d}\phi\left(y\right)
{\bf u_T}\left(dy\right) = \int^{T-t}_0\mathbb{E}\left[\tilde{b} 
\left(s,Y_s;v_s\right)\phi'\left(Y_s\right) + \frac{1}{2} \left(\sigma^2
\left(T-s,Y_s\right)\phi''\left(Y_s\right)\right)\right] ds. 
\end{equation*}
\noindent Fixing $s\in [0,T]$, 
we have
\begin{eqnarray*}
 \mathbb{E}\left[{\tilde b}(s,Y_s;v_s)  \phi'(Y_s) \right] &=&
  \int_\R (\sigma^2(T-s,\cdot)v_{s})'(y) \phi'(y)dy- \int_\R b(T-s,y)
 \phi'(y) v_{s}(y)dy  
\\  
  & =& -\int_\R (\sigma^2)(T-s,y)\phi''(y) v_{s}(y)dy- 
  \int_\R b(T-s,y),\phi'(y) v_{s}(y)dy. 
\end{eqnarray*}
\noindent Hence, we have the identity
\begin{equation*}
\mathbb{E}\left[\phi\left(Y_{T-t}\right)\right] = \int_{\R}\phi\left(y\right){\bf u_T}\left(dy\right) - 
\int^{T-t}_{0}\int_{\R}L_{T-s}\phi\left(y\right)v_{s}\left(y\right)dyds. 
\end{equation*}
\noindent Applying  the change of variable $t \mapsto T-t$, we finally obtain the identity
\begin{equation*}
\int_{\R}\phi\left(y\right)v_{T-t}\left(y\right)dy = \int_{\R}\phi\left(y\right) {\bf u_T}\left(dy\right) - 
\int^{T}_{t}\int_{\R^d}L_{s}\phi\left(y\right)v_{T-s}\left(y\right)dyds.
\end{equation*}
\noindent This means that $t \mapsto u_t$
% $t \mapsto \mathcal{L}\left(Y_{T-t}\right)$
 is a solution of \eqref{EDPTermCond} with terminal value ${\bf u_T}$.

\end{proof}
%We give now a first uniqueness result for equation \eqref{MK}

\begin{rem} \label{RUni}
Precise discussions on existence and uniqueness of \eqref{MK}
are provided in \cite{LucasOR}.
In particular we have the following.
\begin{enumerate}
\item There is at most one solution (in law)  $(Y, v)$ of \eqref{MK}
such that $v$ is locally bounded in $[0,T[ \times \R^d$.
\item There is at most one strong solution 
$(Y, v)$ of \eqref{MK} such that $v$ is locally
Lipschitz in $[0,T[ \times \R^d$.
\end{enumerate}
Item 1. is a consequence of Theorem 10.1.3 of \cite{stroock}.
Item 2. is a consequence of usual pathwise uniqueness arguments
for SDEs.

\end{rem}

\section{Probabilistic representation with jumps for non-conservative PDEs}
\label{sec:jumps}

\setcounter{equation}{0}

%----------------------------------------------------------------------------------------------
In this section, we outline the link between non-conservative PDEs and non-linear jump diffusions.  
This kind of representation was emphasized in~\cite{DelMoral00, DelMoral} to design interacting jump particles systems to approximate time-dependent
 Feynman-Kac measures. For simplicity, we present this correspondence in the simple case of the non-conservative linear
PDE \eqref{epdeIntro0}
%PDE~\eqref{eq:PDE}
when the coefficients do not depend on the solution, see \eqref{epdeIntro0}.
 However, the same ideas could be extended to the non-linear case where the coefficients $\sigma, b,\Lambda$ may depend on the PDE solution. 

%In \cite{JourMeleaWoy}, the authors study existence, uniqueness and particle approximations for these stochastic differential equations. In the case of a symmetric stable driving process, we deduce the existence of a function solution to a nonlinear integro-differential equation involving the fractional Laplacian.

%FIN VERIFICATIONS FRANCESCO.

Let us consider the SDE
\begin{equation}
\label{eq:sde}
\left\{
\begin{array}{l}
dX_t=b(t,X_t)dt+\sigma(t,X_t)dW_t\\
X_0\sim {\bf u_0},
\end{array}
\right .
\end{equation}
where $W$ is a one-dimensional Brownian motion. 
Assume that~\eqref{eq:sde} admits a (weak) solution. 
Let $\Lambda$ be a bounded and negative function defined on $[0,T]\times \R$. 
For any $t\in [0,T]$, we define the measure, $\gamma (t,\cdot)$ such that for any 
real-valued Borel measurable test function $\varphi$
\begin{equation}
\label{eq:gamma}
\int \gamma(t,dx) \varphi(x)= \E\left[\varphi(X_t)\exp\left ( \int_0^t \Lambda(s,X_s)ds\right )\right ]\ .
\end{equation}
We recall that by Section \ref{sec:NonCons} we know that
%Using It\^o's formula, one easily proves that 
 $\gamma$ is a solution (in the distributional sense) of the
linear and non-conservative PDE
\begin{equation}
\label{eq:PDE}
\left\{
\begin{array}{l}
\partial_t \gamma =\frac{1}{2}\partial_{xx}^2(\sigma^2(t,x)\gamma)-\partial_x(b(t,x)\gamma)+\Lambda(t,x)\gamma\\ 
\gamma (0,\cdot)={\bf u_0}.
\end{array}
\right .
\end{equation}
\begin{rem}
If uniqueness of distributional solutions of~\eqref{eq:PDE} holds,
 then $\gamma$ defined
%~
by (\ref{eq:sde},\ref{eq:gamma}) 
%constitutes a probabilistic
 %representation of~\eqref{eq:PDE}.
 %which admits a unique solution denoted by $\gamma$. 
is the unique solution of \eqref{eq:PDE}.
\end{rem}

%%% Equivalence entre (4.3) et (4.5)
%Under the above assumptions, 
%PARLER DE L'EQUIVALENCE
Let $\gamma(t,\cdot)$ be a solution of \eqref{eq:gamma}
which for each $t$ is a positive measure. 
We introduce the family of probability measures $(\eta (t,\cdot))_{t\in [0,T]}$, obtained by normalizing  $\gamma (t,\cdot)$, such that for any real
 valued bounded and measurable test function $\varphi$ we have
\begin{equation}
\label{eq:eta}
\int \eta(t,dx)\varphi(x):=\frac{\int \gamma(t,dx)\varphi(x)}{\int \gamma(t,dx)}\ .
\end{equation}
By simple differentiation of the above ratio and using the fact that $\gamma $ satisfies~\eqref{eq:PDE}, we obtain 
that $\eta$ is a solution in the distributional sense of the integro-differential PDE 
\begin{equation}
\label{eq:IPDE}
\left\{
\begin{array}{l}
\partial_t \eta=\frac{1}{2}\partial_{xx}(\sigma^2(t,x)\eta)-\partial_x(b(t,x)\eta)+ \Big (\Lambda(t,x)-\int \eta(t,dx)\Lambda (t,x)\Big )\eta \\
\eta_0={\bf u_0}\ .
\end{array}
\right .
\end{equation}
%\begin{equation}
%\label{eq:IPDE}
%\left\{
%\begin{array}{l}
%\partial_t \eta_t(\varphi)=\eta_t(b(t,\cdot)\varphi ' )+\frac{1}{2}\eta_t(\sigma^2(t,\cdot)\varphi'')+ \eta_t(\Lambda(t,\cdot)\varphi )-\eta_t(\varphi)\eta_t(\Lambda_t)\\
%\eta_0=\mu_0\ .
%\end{array}
%\right .
%\end{equation}
%
Besides one can express  $\gamma (t,\cdot )$ as a function of $(\eta (s,\cdot ))_{s\in [0,t]}$. Indeed,  since $\gamma$ solves the linear PDE~\eqref{eq:PDE} then in particular approaching the constant
 test function $1$, yields 
$$
\partial_t \int \gamma (t,dx)
=\int \gamma(t,dx)\Lambda (t,x)
=\int \gamma(t,dx)\int \eta(t,dx)\Lambda (t,x)\ ,
$$
which gives  $\int \gamma (t,dx)=\exp\left (\int_0^t \int \eta (s,dx)\Lambda (s,x) ds\right )$. Then by definition~\eqref{eq:eta} of $\eta $, 
\begin{equation}
\label{eq:gammaEta}
\gamma (t,\cdot)=\Big (\int \gamma (t,dx)\Big )\eta (t,\cdot)=\exp\left (\int_0^t \int \eta (s,dx)\Lambda (s,x)ds\right )\eta (t,\cdot)\ .
\end{equation}
We already know that for any solution $\gamma$ of~\eqref{eq:PDE} one can build a solution $\eta$ to~\eqref{eq:IPDE} according to relation~\eqref{eq:eta}.
 Conversely,  for any solution $\eta$ of~\eqref{eq:IPDE}, by similar 
manipulations one can build a solution $\gamma$ of~\eqref{eq:PDE} according to~\eqref{eq:gammaEta}. Hence well-posedness of~\eqref{eq:PDE} is equivalent to well-posedness of~\eqref{eq:IPDE}. 
%Consequently, 
%our two major assumptions ensuring  existence of a weak solution to~\eqref{eq:sde} and uniqueness of distributional solutions to~\eqref{eq:PDE} imply well-posedness of~\eqref{eq:IPDE}.  DEJA DIT

We propose now an alternative probabilistic representation
to \eqref{eq:sde} and \eqref{eq:gamma} of \eqref{eq:PDE}.
Let us introduce the non-linear jump diffusion $Y$ (if it exists),
which evolves between
 two jumps according to the diffusion dynamics~\eqref{eq:sde} and jumps at exponential times with intensity $-\Lambda_t(Y_t)\geq 0$ to a new point independent of the current position and  distributed according to the current law, $\mathcal{L}(Y_t)$. More specifically, we consider a process $Y$ solution of the  following non-linear (in the sense of McKean) SDE with jumps 
\begin{equation}
\label{eq:Jumpsde}
\left\{
\begin{array}{l}
dY_t=b(t,Y_{t^-})dt+\sigma(t,Y_{t^-})dW_t+\int_\R x\mathbf{1}_{\vert x\vert >1} J_t(\mu_{t^-},Y_{t^-},dx)dt+\int_\R x\mathbf{1}_{\vert x\vert \leq 1}  (J_t-\bar J_t)(\mu_{t^-},Y_{t^-},dx)dt\\
Y_0\sim {\bf u_0}\\
\mu_{t^-}=\mathcal{L}(Y_{t^-})\ ,
\end{array}
 \right .
\end{equation}
where $J$ denotes the jump measure and $\bar J$ is the associated predictable compensator such that 
for any probability measure $\nu$ on $\R$
$$
\bar J_t(\nu,y,d(y'-y))=-\Lambda_t(y)\nu(dy')\ ,\quad\textrm{for any}\quad y\,,\,y'\in\R\ .
$$
Note that well-posedness analysis of the above equation constitutes a difficult task. In particular, \cite{JourMeleaWoy}
analyzes well-posedeness and particle approximations of 
%nonlinear SDEs where the Brownian noise is replaced by a Lévy process in~\eqref{eq:MckeanSDE}. Indeed, this setting includes 
some types of non-linear jump diffusions. However, contrarily to~\eqref{eq:Jumpsde}, the nonlinearity considered in~\cite{JourMeleaWoy} is concentrated on the diffusion matrix (assumed to be Lipschitz in the time-marginals of the process w.r.t. Wasserstein metric) and does not involve the jump measure which is assumed to be given. 

Assume that MSKE~\eqref{eq:Jumpsde} admits a weak solution. 
By application of It\^o formula, we observe that the marginals of $Y$ are distributional solutions of~\eqref{eq:IPDE}. 
Indeed, for any real valued test function in $\mathcal{C}^\infty_0(\R)$ 
\begin{eqnarray}
\label{eq:exp}
\E[\varphi (Y_t)]&=&\E[\varphi (Y_0)]\nonumber \\
&&+
\int_0^t\E \left [b(s,Y_{s^-})\varphi'(Y_{s^-})+\frac{1}{2}\sigma^2 (s,Y_{s^-})\varphi''(Y_{s^-})\right ]ds\nonumber \\
&&
+
\int_0^t \E \left [\int \varphi(Y_{s^-}+x)\bar J_s(\mu_{s^-},Y_{s^-},dx)\right ]ds\nonumber \\
&&-\int_0^t\E \left [ \varphi(Y_{s^-})\bar J_s(\mu_{s^-},Y_{s^-},\R)\right ]ds\ .
%%%  CHOIX DE $\nu$.
\end{eqnarray}
\begin{conclusion}
Suppose that \eqref{eq:PDE} admits a unique distributional solution $\gamma$;
let $\eta$ defined by \eqref{eq:eta}. Suppose the existence of a (weak) solution $X$ 
(resp. $Y$) of~(\ref{eq:sde}) 
(resp.~\eqref{eq:Jumpsde}).
%Since we already know that~\eqref{eq:IPDE} admits a unique solution, 
\begin{enumerate}
\item $\eta$ is the unique solution (in the sense of distributions) of 
\eqref{eq:IPDE}.
%Assuming that there exists a solution $Y$ to~\eqref{eq:Jumpsde} implies  that~\eqref{eq:Jumpsde} constitutes a probabilistic representation of~\eqref{eq:IPDE} i.e. 
Moreover $\int_\R \varphi(x) \eta(t,dx) =\E[\varphi(Y_t)], t \ge 0$. 
\item
%To summarize, 
% If assume that (\ref{eq:PDE})
%admits a unique solution $\gamma$ in the sense of distributions,
 We obtain the following  identities for $\gamma$ and $\eta$:
\begin{eqnarray}
\label{eq:yx}
\int \gamma(t,dx)\varphi(x)
&=&\E\big[\varphi(X_t)\exp\left ( \int_0^t \Lambda (s,X_s)ds\right )\big ]\nonumber \\
&=&\exp\left (\int_0^t\int \eta (s,dx)\Lambda (s,x)ds\right )\eta_t(\varphi)\nonumber \\
&=&
\exp\left (\int_0^t\E[\Lambda (s,Y_s)]ds\right )\E[\varphi(Y_t)].
\end{eqnarray}
%where $\gamma$ is characterized either as the unique weak solution of~(\ref{eq:sde},\ref{eq:gamma}) or the unique distributional solution of~\eqref{eq:PDE} and $\eta$ is defined either by~\eqref{eq:eta} or as the unique distributional  solution of~\eqref{eq:IPDE} or as the marginals of $Y$~\eqref{eq:Jumpsde}. 
%Feynman-Kac Particle Integration with Geometric Interacting Jumps Pierre Del Moral , Pierre E. Jacob , Anthony Lee , Lawrence Murray & Gareth W. Peters  Journal Stochastic Analysis and Applications Volume 31, 2013 - Issue 5
\end{enumerate}
\end{conclusion}
Using the above identities allows to design  discrete time interacting particles systems with geometric interacting jump processes. In particular, in~\cite{delmoJacob2013} the authors provide non asymptotic bias and variance theorems w.r.t. the
time step and the size of the system, allowing to numerically approximate the
 time-dependent family of Feynman-Kac measures $\gamma$.

\section{ McKean SDEs in random environment}
\label{random}

\setcounter{equation}{0}

\subsection{The (S)PDE and the basic idea}
\medskip

%% FRANCESCO DENOMINATION $\mu$ DANS LA SECTION3 PLUTOT $u GRAS$?
%%% VERIFIER LA MODIFICATION DE $\mu en \xi
Let $(\Omega, \shf, (\shf_t), \P)$ be a filtered probability space.
We consider a progressively measurable random field
$(\xi(t,x))$. We want to discuss probabilistic
representations of 
 \begin{equation}
\label{E7.1stoch}
\left \{
\begin{array}{lll}
\partial_t u &=& \frac{1}{2} \Delta (\beta(u)) + \partial_t \xi(t,x)
 u(t,x)\,,\quad \textrm{with}\quad \beta(u) =  \sigma^2(u) u. 
\\
%\quad {\rm in} 
%\quad  L^1(\R) \\
u(0,\cdot)& =& {\bf u_0}.
\end{array}
\right.
\end{equation}
%$$\beta(u) =  \sigma^2(u) u. $$
Suppose for a moment that $\xi$ has random realizations
which are smooth in time so that
\begin{equation} \label{ESmoothmu}
 \partial_t \xi(t,x) = \Lambda(t,x;\omega).
\end{equation}
Under some regularity assumptions on $\Lambda$,
\eqref{E7.1stoch} can be observed  as a randomization of
a particular case of the  PDE \eqref{epdeIntro0}.
For each random realization $\omega \in \Omega$, the natural
(double) probabilistic representation is
\begin{equation}
\label{eq:MckeanSPDERep}
\left\{
\begin{array}{l}
Y_t=Y_0+\int_0^t \sigma\Big (u(s,Y_s)\Big )dW_s\\ 
Y_0\,\sim\, {\bf u_0}\\ 
{\displaystyle \int \varphi(x)u(t,x)dx=
\E^\omega\left [\,\varphi(Y_t)\,\,\exp\Big \{\int_0^t\Lambda\big (s,Y_s;\omega)\big )ds\Big \}\,\right]}\ ,\quad \textrm{for}\ t\in [0,T],
\end{array}
\right . 
\end{equation}
where $\E^\omega$ denotes the expectation with frozen $\omega$.
However the assumption \eqref{ESmoothmu} is not realistic and we are
interested in $\partial_t \xi$ being a white noise in time.
Let $N \in \N^*$.
 Let $B^1, \ldots, B^N$ be $N$ independent $(\shf_t)$-Brownian 
motions, $e^1, \ldots, e^N$ be functions in $C^2_b(\R)$.
 In particular they are
 $H^{-1}$-multiplier, i.e. the maps
$\varphi \rightarrow \varphi e^i$ are continuous 
in  $H^{-1}$. 

We define the random field
$ \xi(t,x) = \sum_{i=0}^N e^i(x) B^i_t,$
where $B^0_t \equiv t$
and we consider the  SPDE \eqref{E7.1stoch} in the sense of distributions, i.e. 
\begin{equation}\label{E7.1stochbis}
\int_\R \varphi(x) u(t,x) dx =
\int_\R \varphi(x) {\bf u_0}(dx) +
\halb \int_0^t \int_\R \varphi''(x) \sigma^2(u(s,x)) ds dx 
+ \int_0^t \int_\R \varphi(x) u(s,x) \xi(ds,x) dx,
\end{equation}
where the latter stochastic integral is intended in the It\^o sense.

\subsection{Well-posedness of the SPDE}

%%% FRANCESCO AJUSTER
% Let $(\Omega, \shf, (\shf_t), P)$ be a filtered probability space.
% Let $B^1, \ldots, B^N$ be $n$ independent $(\shf_t)$-Brownian 
% motions, $e^1, \ldots, e^N$ be functions in $C^2_b(\R)$.
%  In particular they are
% %$H^1(\R)$
%  $H^{-1}$-multiplier, i.e. the maps
% $\varphi \rightarrow \varphi e^i$ are continuous 
% in  $H^{-1}$. 

% We define the random field
% $ \mu(t,x) = \sum_{i=0}^N e^i(x) B^i_t,$
% where $B^0_t \equiv t$.
% and we consider the following SPDE in the sense of distributions.
% %% in 
%  \begin{equation}
% \label{E7.1stoch}
% \left \{
% \begin{array}{ccc}
% \partial_t u &=& \frac{1}{2} \Delta (\beta(u)) + \partial_t \mu(t,x)
%  u(t,x). \\
% %\quad {\rm in} 
% %\quad  L^1(\R) \\
% u(0,\cdot)& =& {\bf u_0}.
% \end{array}
% \right.
% \end{equation}
% $$\beta(u) =  \Phi^2(u) u. $$
The theorem below contains results taken from 
 \cite{barbuPorousSPDE, Roeckner2016}.

\begin{thm}  \label{th3.1}
 Suppose that
% $e^i \in W^{1,\infty}, 1 \le i \le N$ and 
$\beta$ is  Lipschitz.
%We obtain  a (double, weak-strong probabilistic representation 
%of 
\begin{itemize}
\item Suppose that $u_0 \in L^2(\R)$.
There is a
 solution  to equation \eqref{E7.1stoch}.
%strong
% This solution satisfies
%$$\E\left[\sup_{t\in[0,T]}|X(t)|^2_2\right]
%\le 2|x|^2_2e^{3C^2_\infty t}.$$In particular, $X\in L^2(\Omega;L^\infty([0,T];L^2(\R^d)))$.
\item  Assume further that $\beta$ is non-degenerate, i.e.
%\begin{equation}\label{e3.5}
$\beta(r) \ge a r^2, \ r \in \R,$
%\end{equation}
where
$a>0$.
 Then, there is a  solution $u$ to \eqref{E7.1stoch}
for any probability  ${\bf u_0}(dx)$ (even in $H^{-1}(\R)$).
\item There is at most one solution in the class of random fields
$u$ such that 
$\int_{[0,T] \times \R} u^2(t,x) dtdx < \infty$ a.s.
\end{itemize}
\end{thm}
\begin{rem} 
\begin{itemize}
\item  Previous result extends to the case of an infinite number
of modes $e^i$ and for $d \ge 1$.
\item We remark that the $\partial_t \xi(t,x)$ is a coloured
noise (in space). The case of space-time white noise seems very difficult 
to treat.
% \item Partial generalizations of previous results to 
% a porous media type equation with first order term 
% is \cite{BarbuFokker}.

\end{itemize}
\end{rem}

%-----------------------------------------------------------

\subsection{McKean equation in random environment}

Given a local martingale $M$,
$\she(M)$ denotes the Dol\'eans exponential of $M$ i.e.
$ \exp(M_t - \frac{1}{2}[M]_t), t \ge 0$.
We say that a filtered probability space $(\Omega_0, \shg, (\shg_t), Q)$ 
is a {\bf suitable enlarged space} of  $(\Omega, \shf, (\shf_t), P)$,
if the following holds.
\begin{enumerate}
\item There is a measurable space $(\Omega_1, \shh)$
with $\Omega_0 = \Omega \times \Omega_1$, $\shg = \shf \otimes \shh$
and a random kernel  $(\omega, H) \mapsto  \Q^\omega(H)$
defined on $\Omega \times \shh \rightarrow [0,1]$
such that the probability $\Q$ on $(\Omega_0,
\shg)$ is defined by 
$d\Q(\omega, \omega_1) = d\P(\omega) \Q^\omega(\omega_1)$.
\item The processes $B^1, \ldots, B^N$ are $(\shg_t)$- Brownian motions
where $\shg_t = \shf_t \vee \shh$.
\end{enumerate}

%To  a measurable process $Y$  on $(\Omega_0,\shg)$, we will associate its
%{\bf family of $\mu$-marginal weighted laws}, 
%i.e.
% the family of random kernels ($t \in [0,T]$),
% \begin{equation}
% \label{1.2ter}
% \varphi   \mapsto E^{Q^\omega} \left( \varphi (Y_t (\cdot,\omega))
%  \she_t \left(\int^\cdot_0 \mu (d s, Y_s) (\cdot , \omega)\right) \right),
% %= \int_\R \varphi(r) \Gamma_t^Y(\d r, \omega),
% \end{equation}
%where $\varphi$ is a generic bounded real Borel function. We will also say that for fixed $t \in [0,T], \; \Gamma_t$ is {\bf the $\mu$ marginal weighted law}
% of $Y_t$.

\begin{defi} \label{DWeakExistence}
We say that the {\it non-linear doubly-stochastic diffusion} 
\begin{equation}
\label{DPIY}
\left \{
\begin{array}{lll}
Y_t &=& Y_0 + \int_0^t \Phi(u(s,Y_s)) dW_s, \\
%\xi-{\rm Weighted \ Law } (Y_t) &=& u(t,x ) dx    , \quad t \in ]0,T],\\
\int  \varphi(x) u(t,x) dx &=& 
 \E^{\Q^\omega} \left( \varphi (Y_t (\omega, \cdot)) 
 \she_t \left(\int_0 \xi (d s, Y_s) (\omega, \cdot)\right) \right),
\\
\xi-{\rm  Law } (Y_0) &=& {\bf u_0}(dx), 
\end{array}
\right.
\end{equation}
admits {\bf weak existence} on $(\Omega,\shf,(\shf_t),\P)$ if there is a
suitably enlarged probability space $(\Omega_0, \shg, (\shg_t), \Q)$
an $(\shg_t)$-Brownian motion $W $ such that \eqref{DPIY} is verified.
The couple $(Y,u)$ will be called {\bf weak solution} of \eqref{DPIY}.
\end{defi}
\begin{rem} \label{RDoubleIto}
\begin{itemize}
\item We remark that the second line in \eqref{DPIY}
represents a sort of {\bf $\xi$-marginal weighted law}
of $Y_t$.
\item Let $(Y,u)$ be a solution to \eqref{DPIY}.
Then $u$ is a solution to \eqref{E7.1stoch}.
\item Such representation allows to show
that  $u(t,x) \ge 0$,  $d\P  dt dx$ a.e. 
and, at least if $e^0 = 0$, 
$ \E\left(\int_\R u(t,x) dx\right) = 1,$ 
so that the conservativity is maintained at the expectation level.
\end{itemize}
\end{rem}

\begin{defi} \label{DWeakUniqueness}
Let two measurable 
 random fields $u^i: \Omega \times [0,T] \times \R \to \R,
 i = 1,2$ on $(\Omega, \shf, \P, (\shf_t))$,
  and $Y^i$, on a suitable extended probability space 
$(\Omega_0^i, \shg^i, (\shg^i_t),  \Q^i), i=1,2$,
such that $(Y^i, u^i)$  are (weak) solutions of \eqref{DPIY}
on $(\Omega, \shf, (\shf_t), \P)$.
If we  always have that $(Y^1, B^1, \ldots, B^N)$ 
and   $(Y^2, B^1, \ldots, B^N)$ have the same law, then 
we say that \eqref{DPIY} admits
 {\bf weak uniqueness}
(on $(\Omega, \shf, (\shf_t),  \P)$).
\end{defi}

\begin{thm} Under the assumption of  Theorem \ref{th3.1}
equation \eqref{DPIY} admits (weak) existence and uniqueness 
on $(\Omega, \shf, (\shf_t), \P)$.
\end{thm}

\section{McKean representation of stochastic control problems}
\label{sec:control}
%-----------------------------------------------------------
\subsection{Stochastic control problems and non-linear Partial Differential Equations}
%-----------------------------------------------------------

\setcounter{equation}{0}

Let us briefly recall the link between stochastic control and non-linear
 PDEs given by the Hamilton-Jacobi-Bellman (HJB) equation. 
We refer for instance to~\cite{touzibook, Pham09, gozzibook} for more details. 
Consider a \textit{state process} $(X_{s}^{t_0,x,\alpha})_{t_0\leq s \leq T}$ on $[t_0,T]\times \R^d$ 
solution to the controlled SDE
\begin{equation} 
\label{eq:DynX}
\left\{
\begin{array}{lll}
dX^{t_0,x,\alpha}_s&=&b\big (s,X^{t_0,x,\alpha}_s,\alpha(s,X^{t_0,x,\alpha}_s)\big )ds+\sigma \big (s,X^{t_0,x,\alpha}_s,\alpha(s,X^{t_0,x,\alpha}_s)\big )dW_s \\ 
X_{t_0}^{t_0,x,\alpha}&=&x\ ,
\end{array}
\right .
\end{equation}
where $W$ denotes the Brownian motion on $[t_0,T]\times \R^d$, and $\alpha(s,X^{t_0,x,\alpha}_s)$ represents \textit{Markovian} control in the sense that the control at time $t$ is supposed here to depend on $t$ and on the current value of the state process:  
\begin{equation}
\label{eq:feedback}
\alpha\in \mathcal{A}_{t_0,T}:=\left\{\alpha\,: (t,x)\in [t_0,T]\times \R^d\, \mapsto\,  \alpha(t,x)\in A \subset \R^k\right \}\ ,
\end{equation}
$A$ being a subset of $\R^k$. 
%We are interested in  maximizing a criteria $J$,  for a given initial time and state $(t_0,x)\in [0,T]\times\mathbb{R}^{d}$, over the Markovian controls, $\alpha \in \mathcal{A}_{t_0,T}$
For a given initial time and state $(t_0,x)\in [0,T]\times\mathbb{R}^{d}$,
 we are interested in maximizing, over the Markovian controls $\alpha \in \mathcal{A}_{t_0,T}$, the criteria 
\begin{equation}
\label{eq:probJ}
J(t_0,x,\alpha):=
\, 
\mathbb{E} \left[ g(X_{T}^{t_0,x,\alpha})+ \int _{t_0}^{T} f\big (s,X_{s}^{t_0,x,\alpha},\alpha(s,X_s^{t_0,x,\alpha})\big )ds   \right]
\ .
\end{equation}
In the above criteria, the function $f$ is called the \textit{running gain} whereas $g$ is called the \textit{terminal gain}. 
\begin{rem}
At first glance, the set of control processes of the form $\alpha_t=\alpha(t,X_t)$ defined in~\eqref{eq:feedback} may appear too restrictive compared to a larger set of non-anticipative controls $(\alpha_t)$ which may depend on all the past history of the state process $(X_t)$. However, in the framework of Markov control problems (for which the state process $(X_t^{t_0,x,a})$ is Markov, as soon as the control is fixed to a deterministic value $\alpha_t=a\in A$, for all $t\in[t_0,T]$), it is well-known that the optimal control process $(\alpha_t)$ lies in the set of Markovian controls verifying $\alpha_t=\alpha(t,X_t)$. Hence, considering controls of the particular form~\eqref{eq:feedback} is done here without loss of generality. 
\end{rem}
To tackle this finite horizon stochastic control problem, the usual approach consists in introducing the associated value (or \textit{Bellman}) function  $v$ : $[t_0,T] \times \mathbb{R}^{d}\rightarrow \R$ representing the maximum gain one can expect, starting from time $t$ at state $x$, i.e.
\begin{equation}
\label{eq:prob}
v(t,x):= \displaystyle {\sup_{\alpha \in \mathcal{A}_{t,T}} }
\, 
J(t,x,\alpha)
\ ,\quad\textrm{for}\ t\in [t_0,T]\ .
\end{equation}
Note that the terminal condition is known, which fixes $v(T,x)=g(x)$, whereas the initial condition $v(t_0,x)$ corresponds to the solution of the original minimization problem.
The value function is then proved to verify the  \textit{Dynamic Programming Principle (DPP)}   which consists in the  backward induction
\begin{equation}
\label{eq:DPP}
v(t,x)=\sup_{\alpha\in \mathcal{A}_{t,\tau}} \E\big [\,\int_t^{\tau}f(s,X^{t,x,\alpha}_s,\alpha(s,X^{t,x,\alpha}_s))ds+v(\tau,X^{t,x,\alpha}_{\tau})\,\big ]\ ,\quad \textrm{for any stopping time}\ \tau \in ]t,T]\ .
\end{equation}
Under continuity assumptions on $b$, $\sigma$, $f$, $g$, using DPP together with It\^o formula
 allows to characterize  $v$ as a viscosity solution of the  HJB equation 
\begin{equation}
\label{eq:HJB}
\left\{
\begin{array}{lll}
v(T,x)=g(x)\\ 
\partial_t v (t,x)+H(t,x,\nabla v(t,x),\nabla^2 v(t,x)) =0
\ ,
\end{array}
\right .
\end{equation}
where $\nabla$ and $\nabla^{2}$ denote the gradient and the Hessian operators and the so-called, \textit{Hamiltonian},  $H$ denotes the real valued function defined on $[0,T]\times \R^d\times \R^d\times \mathcal{S}^d$   ($\mathcal{S}^d$ denoting the set of symmetric matrices in $\mathbb{R}^{d\times d}$), such that 
\begin{equation}
\label{eq:H1}
H(t,x,\delta,\gamma):=\sup_{ a\in A} \left \{f(t,x, a)+b(t,x, a)^\top \delta(t,x)+\frac{1}{2} Tr[\sigma \sigma ' (t,x, a) \gamma (t,x)]\right \}\ .
\end{equation}
Note that~\eqref{eq:HJB} is a non-linear PDE because of the nonlinearity in the Hamiltonian induced by the supremum operator. 
Besides, assuming that, for all $(t,x)\in [t_0,T]\times \R^d$, the supremum in~\eqref{eq:H1} is attained at a unique maximizer, then the optimal control $\alpha^*$ is directly obtained as a function of the Bellman funtion and its derivatives, i.e.
\begin{equation}
\label{eq:H2}
\alpha^*(t,x)=\textrm{arg}\max_{ a\in A} \left \{f(t,x, a)+b(t,x, a)^\top \nabla v(t,x)+\frac{1}{2} Tr[\sigma \sigma^\top (t,x, a) \nabla ^2v (t,x)]\right \}\ .
\end{equation}
Except in some very concrete cases such as the Linear Quadratic Gaussian (LQG) setting (where the states dynamics involve an affine drift with Gaussian noise and the cost is quadratic both w.r.t. the control and the state), there is no explicit solution to stochastic control problems. 
To numerically approximate the solution of equation~\eqref{eq:HJB}, several approaches have been proposed, mainly differing in the way the value function $v$
 is interpreted.  Indeed, as pointed out,  $v$ can be viewed either as  the solution to the control problem~(\ref{eq:prob}),
or as a (viscosity) solution of the non-linear PDE~(\ref{eq:HJB}).
\begin{enumerate}
\item
When $v$ is defined as the solution to the control problem~(\ref{eq:prob}), a natural approach consists in discretizing %the DPP equation~\eqref{eq:DPP}. 
%	One can consider a time discrete approximation of 
	the time continuous control problem and apply the time discrete Dynamic Programming Principle~\cite{BertShre}. Then the problem consists in maximizing over the controls, backwardly in time, the conditional expectation of the value function related to \eqref{eq:DPP}. The maximization at time step $t_k$ can be done via a parametrization of the control $x\mapsto \alpha^{\theta}_{t_k}(x)$ 
via a parameter $\theta$ so 
that parametric
 optimization methods such as the stochastic gradient algorithm could be applied to maximize the expectation over $\theta$. It remains to approximate the conditional expectations  by numerical methods such as PDE, Fourier, Monte Carlo, Quantization or lattice methods\ldots  A great variety of numerical approximation schemes have been developed in the specific Bermudan option valuation test-bed~\cite{glasserman, longstaff, BallyEtal05, TV, Delmoral2011, bouchard-warin}.
Alternatively, one can use Markov chain approximation method~\cite{Kushner1992} which consists in a  time-space discretization 
%obtained  by renormalizing the finite differences scheme such as 
designed to obtain a proper Markov chain.
% One advantage of this approach is that  the probabilistic setting helps to rely on probabilistic arguments to prove the convergence of the numerical scheme.
\item 
In the second approach we recall that $v$ is viewed
as the solution of \eqref{eq:HJB}.
The problem amounts then to discretize a non-linear PDE. Then one can rely on  numerical analysis methods (e.g. finite differences, or finite elements) and use monotone approximation schemes in the sense of Barles and Souganidis~\cite{Barles1991} to build converging approximation schemes, e.g.
~\cite{zidani,Forsyth2012}.  This type of approach is in general limited to state space dimension lower than $4$.  To tackle higher dimensional problems, one approach consists in converting the PDE into a probabilistic setting in order to apply Monte Carlo types algorithms. To that end, various kinds of probabilistic representations of non-linear PDEs are available. 
Forward Backward Stochastic Differential Equations (FBSDE) were introduced 
in~\cite{PardouxEtPeng92} as probabilistic representations of semi-linear PDEs.
Then various types of numerical schemes for FBSDE have been developed. They mainly 
differ in the approach of evaluating conditional expectations: 
\cite{BouchardTouzi} (resp. \cite{GobetWarin},
% perform regression methods, 
\cite{DelarueMenozzi, pages_quant}) use kernel (resp. regression, quantization) methods. 
Recently, important progresses have been done performing machine learning techniques, see e.g. 
\cite{jentzen, hur2019machine}.
%%%%%%%%%%%%%%%
Branching processes~\cite{Mckean75, HOTTW} can also provide probabilistic representations of semi-linear PDEs via Feynman-Kac formula. 
Non-linear SDEs in the sense of McKean~\cite{Mckean} are another approach that constitutes the subject of the present paper.
\item
Other approaches take advantage of both interpretations see for instance~\cite{touzi} 
%or \ref{eq:prob}) as
and in~\cite{Tan2012}.
\end{enumerate}
%Mixing the above approaches, one can discretize the time, use conditional expectation approximation schemes to compute the expected value function and its derivatives and minimize the Hamiltonian~
%\eqref{eq:F}
%CETTE FORMULE N'EXISTE PAS \\
% to obtain the optimal control as a function of the value function derivatives without dealing with the difficult task of solving the supremum over $\mathcal{A}$. Then $v$ can be interpreted as a viscosity solution of~(\ref{eq:HJB}) as in~\cite{touzi}, or as the solution to the control problem~(\ref{eq:prob}) as in~\cite{Tan2012}.

\subsection{McKean type representation in a toy control problem example}
%-----------------------------------------------------------

In order to illustrate the application of MFKEs to control problems,
we consider a simple example corresponding to an inventory problem, for which the Hamiltonian 
maximization~\eqref{eq:H1},\eqref{eq:H2} is explicit. The state $(X_t)_{t\in[t_0,T]}$ denotes the stock level evolving randomly with a control of the drift $\alpha$:
$$
\left\{
\begin{array}{l}
dX_t^{t_0,x,\alpha}
=
-
\alpha(t,X_t^{t_0,x,\alpha})dt
+\sigma dW_t
\\
J(t_0,x,\alpha)=\displaystyle{\sup_{\alpha\in \mathcal{A}_{t_0,T}}} \, 
\E\Big [
g(X_T^{t_0,x,\alpha})
-
\int_{t_0}^T\Big [
\big (\alpha(t,X_t^{t_0,x,\alpha})-D_t\big )^2
+
h(X_t^{t_0,x,\alpha})
\Big ]
\,dt\Big ].
\end{array}
\right .
$$
Bound constraints on the storage level are implicitly forced by the penalization $h$. A target terminal level is indicated by the terminal gain  
$g$, supposed here  to be Lebesgue integrable. The objective is then to follow a deterministic target profile $(D_t)_{t\in [0,T]}$, on a given finite horizon $[t_0,T]$. When the admissible set in which the controls take their values $A=\R$, one can explicitly derive the optimal control as a function of the value function derivative 
$$
\alpha^\ast(t,\cdot)=D_t+\frac{1}{2}(\partial_xv)(t,\cdot)\ ,
$$
which yields the following HJB equation
$$
\partial_tv+\frac{1}{4} (\partial_xv)^2+D_t\partial_xv+\frac{\sigma^2}{2}\partial_{xx}v-h=0\ .
$$
Reversing the time, (with $t_0=0$) gives $(t,x)\mapsto u(t,x):=v(T-t,x)$ solution of 
\begin{equation}
\label{eq:KPZ}
\left \{  
\begin{array}{lll}
 \partial_t u &=&  \frac{1}{4}( \partial_x u )^2+
\frac{\sigma^2}{2} \partial_{xx}u+D_t\partial_xu-h, \\
%\ \quad \quad\textrm{for}\ (t,x) \in [0,T] \times \R}, \\
u(0,x) &= &g(x).
\end{array}
\right .
\end{equation}
We recover the framework of~\eqref{epdeIntro0}, with $\Lambda(t,x,y,z)=\frac{1}{4}\frac{\vert z\vert^2}{y}-\frac{h(x)}{y}$ and $b(t,x,y)=-D_t$. 
Consequently the Bellman function  $v$ can be represented via
\begin{equation}
\label{eq:MckeanExtIntroControl}
\left\{
\begin{array}{l}
Y_t=Y_0+  \sigma W_t 
- \int_0^t D_s ds\\ 
Y_0\,\sim\, \frac{g(x) dx}{\int_\R g(y)dy}
\\ 
\int \varphi(x)v(t,x)dx= \left(\int_\R g(y) dy\right) \quad 
\E\left [\,\varphi(Y_{T-t})\,\,\exp\Big \{\int_t^{T}\Lambda\big (s,Y_{T-s},{v}(s,Y_{T-s}),\nabla{v}(s,Y_{T-s})\big )ds\Big \}\,\right], \\
\textrm{for}\ t\in [0,T].
\end{array}
\right . 
\end{equation}

\section*{Acknowledgments}

The work was supported by a public grant as part of the
{\it Investissement d'avenir project, reference ANR-11-LABX-0056-LMH,
  LabEx LMH,}
in a joint call with Gaspard Monge Program for optimization, operations research and their interactions with data sciences.

\bibliographystyle{plain}
%\bibliography{ThesisLucasOLD.bib}
\bibliography{../../../BIBLIO_FILE/ThesisLucas.bib}

\end{document}